\documentclass[reqno]{amsart}
\usepackage[backend=bibtex, sorting=none, bibstyle=alphabetic, citestyle=alphabetic, sorting=nyt, maxbibnames=99, giveninits=true, isbn=false,url=false, doi=false, eprint=false]{biblatex}  
\bibliography{references.bib}
\usepackage{amssymb, calrsfs, graphics, graphicx, enumerate, enumitem, url, xcolor, hyperref}
\usepackage[ruled, lined, linesnumbered, longend]{algorithm2e}
\usepackage[justification=centering]{caption}
\usepackage{tikz}
\usepackage{comment}
\usetikzlibrary{positioning, arrows, decorations.markings, calc}
\newtheorem{theorem}{Theorem}[section]
\newtheorem{lemma}[theorem]{Lemma}

\numberwithin{equation}{section}

\theoremstyle{remark}
\newtheorem*{remark}{Remark}

\newcommand{\RE}{\mathrm{Re}}

\title{An explicit version of Carlson's theorem}
\author[S. Chourasiya]{Shashi Chourasiya}
\address{School of Science\\
The University of New South Wales, Canberra, Australia}
\email{s.chourasiya@unsw.edu.au}
\date\today
\keywords{Zero density estimate, Approximate functional equation, Riemann zeta-function.}
\subjclass[2020]{Primary 11N56, 11N37 Secondary 11M06}

\begin{document}

\begin{abstract}
Let $N(\sigma,T)$ denote the number of nontrivial zeros of the Riemann zeta function
with real part greater than $\sigma$ and imaginary part lying between $0$ and $T$. In this article, we provide an explicit version of Carlson's zero density estimate, that is, $N(\sigma, T) \leq 0.78 T^{4 \sigma (1- \sigma)}  (\log T)^{5-2 \sigma} $, with a slight improvement in the exponent of the logarithm factor. 
\end{abstract}

\maketitle
 
\section{Introduction}
Let $\zeta(s)$ denote the Riemann zeta function and $\rho$ denote a non-trivial zero in the critical strip $0<\RE(s)< 1$. For $\frac{1}{2} \leq \sigma \leq 1$ and $T>0$, let 
\begin{align*}
    N(\sigma,T):= \# \{ \rho=\beta+i \gamma: \zeta(\rho)=0, \beta> \sigma, 0 < \gamma \leq T\},
\end{align*}
where the zeta zeros are counted with multiplicities. While the Riemann hypothesis asserts that $N(\sigma,T)=0$ for $\sigma \geq 1/2$, weaker bounds of $N(\sigma,T)$ play a crucial role in several applications of the theory of the zeros of $\zeta(s)$ such as study of primes in short intervals. To bound $N(\sigma,T)$ is generally referred to as the zero density estimate of $\zeta(s)$. In $2021,$ Platt and Trudgian \cite{PT21} verified that $N(\sigma,T)=0$, for $T \leq H_{0}=3\cdot 10^{12}$. Consequently, we can restrict ourselves to the range of $T\geq3 \cdot 10^{12}$ throughout the article.

Depending on the range of $\sigma$, various zero-density estimates exist, for instance the first one was provided by Bohr and Landau \cite{BL13} in 1914, as follows:
\begin{align}
    N(\sigma, T) = O \left( \frac{T}{\sigma- \frac{1}{2}}\right),
\end{align}
for $T$ large, and Carlson then improved the bound in \cite{Ca21} to:
\begin{align}
    N(\sigma,T) =O (T^{4 \sigma(1-\sigma)} \log^{4}T).
\end{align}
 In $1937$, considering a bound of the zeta function on the critical line $ \zeta\left(\frac{1}{2}+it\right) = O(t^{c+\epsilon}), $ Ingham \cite{ingham1937difference} proved 
\begin{align} 
    N( \sigma,T)=O \left( T^{(2+4c)(1-\sigma)} \log^{5} T \right).\label{Ingh37}
\end{align}
When $\sigma$ is close to the line $\frac{1}{2}$-line, Selberg \cite{Sel46} showed that
\begin{align}
    N(\sigma,T)= O(T^{1-\frac{1}{4}(\sigma-\frac{1}{2})} \log T). \label{Sel46}
\end{align}

  Some other notable contributions are \cite{Tur61}, \cite{Huxley1971OnTD}, \cite{Huxley1975LargeVO}, \cite{Jutila1977ZerodensityEF}, \cite{jean_bourgain_2000} and \cite{guth2024new}.

Several non-explicit zero-density estimates are known, however, less is known about explicit versions. In $2014$, Kadiri \cite{Kad14shortint} provided an explicit version of Bohr and Landau's estimate and around the same time, Ingham's bound \eqref{Ingh37} was made explicit by Ramar\'e \cite{Ram2015} and further improved by Kadiri, Lumley and Ng \cite{KADIRI201822} as:
\begin{align}
    N(\sigma,T) \leq C_{1} T^{\frac{8}{3}(1-\sigma)} \log^{5-2 \sigma}T + C_{2} \log^{2} T, \label{KLN}
\end{align}
where $C_{1}$ and $C_{2}$ are absolute constants depending on $\sigma$. Simoni\v{c} \cite{Simoni2019ExplicitZD} gave the following explicit version of \eqref{Sel46}:
\begin{align}
   N(\sigma,T) \leq \frac{K_{2}}{2^{1-\frac{1}{4}\left(\sigma-\frac{1}{2}\right)}} T^{1-\frac{1}{4}\left(\sigma-\frac{1}{2}\right)}\log \frac{T}{2} \label{Sim}
\end{align}
for some constant $K_{2}$ and $T \geq 2H_{0}$.
Moreover, near to the $1$-line, Bellotti gave two explicit results, one is log-free \cite{bellotti2024explicit} another one is Korobov-Vinogradov type estimate \cite{bellotti2023explicit}. However, an explicit version of Carlson's bound is not known yet.

The aim of this paper is to provide the first explicit version of Carlson's bound in Theorem \eqref{mainth}. Furthermore, some values of the constant is listed in the Table \ref{table1} while varying $T_{0}$ and $\sigma$.





    
\begin{theorem}{\label{mainth}}
    For $T\geq T_{0} \geq 3 \cdot 10^{12}$ and $\sigma \geq 0.6$, the following estimate holds:
    \begin{align}
        N(\sigma,T) \leq K(\sigma,T_{0}) T^{4 \sigma(1-\sigma)}\log^{5-2 \sigma}T,
    \end{align}
where $K(\sigma,T_{0})$ is defined in \eqref{const}.
\end{theorem}
Our result improves the Carlson's estimate asymptotically as we get the exponent $5-2\sigma$ of $\log T$ instead of $4$.

The fact that the Carlson method uses some basic techniques, will reflect in the respective generalization for other $L$-functions as well. Some values of $T_{0}$ and $K(\sigma,T_{0})$ are given below:
\begin{center}
\textbf{Table $1$} \label{table1}\\
\vspace{0.2cm}
\begin{tabular}{ |c|c|c| } 
 \hline
 $T_{0}$ & $\sigma$  & $K(\sigma,T_{0})$ \\ 
 \hline
 $3 \cdot 10^{12}$ &$0.60$ &  $0.7756$ \\ 
  $3 \cdot 10^{12}$ & $0.66$&  $0.2781$ \\ 
 $10^{20}$ & $0.60$ &  $0.0686$\\
 $10^{20}$ & $0.65$ &  $0.0963$\\
 $10^{50}$ & $0.61$ &  $0.0597$ \\
  $10^{50}$ & $0.65$ &  $0.0453$ \\
$10^{70}$ & $0.62$ &  $0.2447$\\
$10^{70}$ & $0.66$ &  $0.0253$\\
$10^{200}$ & $0.62$ &  $0.1414$\\
  \hline
  \end{tabular}
  \end{center}
  \vspace{0.3cm}
 Here, the constant $K(\sigma,T_{0})$ approaches $\frac{1}{2 \pi} \left( \frac{0.68 \sigma (2 \sigma-1)^{2}}{1-0.5^{4 \sigma(1-\sigma)}}\right)$ when $T_{0}$ is sufficiently large.

  \subsection{Comparison with existent estimates}

Simoni\v{c}'s result \eqref{Sim} gives the best bound for $T \geq 6\cdot 10^{12}$ and $\sigma \in (1/2, 0.64)$. Bellotti's two estimates are sharpest, one for the range $3 \cdot 10^{12}\leq T \leq \exp(6.7\cdot 10^{12})$ when $\sigma \in (0.985,1],$ and the other for $T \geq \exp(6.7 \cdot 10^{12})$ when $\sigma \in (0.98,1]$. For the rest of the region in the critical strip, the bound given by Kadiri, Lumley and Ng is the sharpest one. For Theorem \ref{mainth}, we consider the region $\sigma \in [0.6,2/3]$, and hence the only existing estimates we need to look for the comparison are due to Simoni\v{c} and Kadiri, Lumley and Ng. If we fix $T_{0}=3 \cdot 10^{12}$, then on comparing numerically, the Theorem \ref{mainth} gives better bound than \eqref{Sim} for $T \geq 9.48 \cdot 10^{308}$ when $\sigma \in [0.6,2/3]$. If we increase the value of $T_{0}$, for instance, $T_{0}=10^{200}$, then Theorem \ref{mainth} becomes sharper for $\sigma \in [0.6,2/3]$ and $T \geq 1.43 \cdot 10^{236}$.

If we compare the exponents of $T$, Theorem \ref{mainth} should be a better estimate than \eqref{KLN}, for $\sigma \leq 2/3$ when $T$ is large. However, if this method were further improved, then there would be some scope of improvement for $\sigma > 2/3$ as well, in a finite region.

Below is the diagram demonstrating the regions where some existing zero density estimates including Theorem \ref{mainth} are sharpest,
 for the range $\sigma \in [0.5, 0.8]$:
\vspace{0.4cm}

\begin{tikzpicture}
    \draw[->] (0,0) -- (7,0) node[below] {Re($s$)};
    \draw[->] (0,0) -- (0,7) node[left] {Im($s$)};

    
    \draw[dashed] (0,3) -- (5,3) node[pos=0,left] {$T = 1.43 \cdot 10^{236}$};
    \draw[] (0,0) -- (5,0) node[pos=0,left] {$T = 3 \cdot 10^{12}$};
    
    \fill[blue!20] (2.8,3) rectangle (4.4,5);
    \fill[gray!20] (1.2,2) rectangle (2.8,5);
    \fill[gray!20] (2.8,2) rectangle (4,3);
     \fill[pink!20] (1.2,0) rectangle (2.8,2);
     \fill[pink!20] (2.8,0) rectangle (4.4,2);
     \fill[pink!20] (4,0) rectangle (4.4,3);
     \fill[pink!20] (4.4,0) rectangle (6,5);
    \draw[dashed, red] (1.2,-0.5) -- (1.2,2.5);
    \node at (1,-0.3) {$\frac{1}{2}$};
     \draw[black,dashed] (4.4,0)--(4.4,3);
    \draw[black,dashed] (2.8,0)--(2.8,3);
    
     \draw[black] (2.8,3)--(2.8,5);
    

    \draw[black] (4.4,3)--(4.4,5);
     \draw[black] (2.8,3)--(4.4,3);
     \draw[black,dashed] (4,0)--(4,2);
     \draw[black] (1.2,2)--(1.2,5);
     \draw[black] (1.2,2)--(4,2);
     \draw[black] (4,2)--(4,3);
    \node at (4.4,-0.3) {$\frac{2}{3}$};
    \node at (4,-0.3) {$\frac{37}{58}$};
    \node at (2.8,-0.3) {0.6};
     \node at (6,-0.3) {0.8};
    \fill[red] (1.2,0) circle (1pt);
 \fill[black] (2.8,0) circle (1pt);
  \fill[black] (4,0) circle (1pt);
   \fill[black] (4.4,0) circle (1pt);
    \fill[black] (6,0) circle (1pt);
    \fill[gray!20] (1.1, 5.9) rectangle (1.3,6.2);
    \node[right] at (1.4, 6.1) {$N(\sigma, T) \leq K'_{2} T^{1- \frac{1}{4} \left( \sigma-\frac{1}{2}\right)} \log (T/2)$};
    
    \fill[blue!20] (1.1, 6.3) rectangle (1.3,6.6);
    \node[right] at (1.4, 6.5) {$N(\sigma, T) \leq K(\sigma,T_{0})T^{4\sigma(1 - \sigma)} (\log T)^{5-2\sigma}$};
    
    \fill[pink!20] (1.1,6.7) rectangle (1.3,7);
    \node[right] at (1.4,6.9) {$N(\sigma, T) \leq K_{1} T^{\frac{8}{3} (1 - \sigma)} (\log T)^{5 - 2\sigma}$};
    
\end{tikzpicture}
\begin{center}
    Figure 1
\end{center}
\vspace{0.3cm}

Our approach to the main theorem strongly relies on Carlson's original proof. The structure of this paper is as follows. In $\S$ \ref{Litsect}, we give the classical method of Littlewood to count the zeros. Further, in $\S$ \ref{presect}, some explicit results are provided. The key idea is to obtain a second moment of the mollified zeta function in $\S$ \ref{Mollsect}, on which this zero density estimate is heavily based. The proof of Theorem \ref{mainth} is given in $\S$ \ref{zeroden}.

\section{Littlewood's classical method to count the zeros \label{Litsect}}

In this section, we set up the proof for the zero density estimate. We begin with the assumption that $T$ is not the ordinate of a zero of zeta function. Let $F(s)$ be an entire function. Consider the meromorphic function $h(s)=\zeta(s) F(s)$ and 
\begin{align}
    N_{h}(\sigma, T)= \# \{ \rho' = \beta' + i \gamma' : h(\rho') =0, 
 \sigma< \beta' <1 \ \ \text{and}  \ \ 
 0<\gamma' <T\}.
\end{align}
Let $H_{0}=3\cdot 10^{12}$ be defined as in the previous section. Observe that, for a parameter $H \in ( 0, H_{0})$, we have $N(\sigma,H)=0$. Subsequently, for $T \geq H_{0}$, using the fact $N(\sigma,T) \leq N_{h}(\sigma,T)$, we can write
\begin{align*}
    N( \sigma, T)= N(\sigma,T)-N(\sigma,H)\leq N_{h}(\sigma,T)- N_{h}(\sigma,H).
\end{align*}
On comparing to its average, we have
\begin{align}
    N_{h}(\sigma,T)- N_{h}(\sigma,H) \leq \frac{1}{\sigma- \alpha}\int_{\alpha}^{\beta} (N_{h}(\tau,T)- N_{h}(\tau,H)) \text{d}\tau, \label{Litlemma}
\end{align}
where $\beta >1$ and $\alpha$ is a parameter satisfying $\frac{1}{2}< \alpha < \sigma$. Here $h(s)\neq 0$ on  $\sigma=\beta$ and hence the function $\log h(s)$ is regular in the neighbourhood of $\sigma=\beta.$ Now consider a rectangle $\mathcal{R}$ with vertices $\alpha +iH, \beta +iH, \beta +iT,$ and $\alpha +iT$ and apply the classical lemma of Littlewood \cite{titchmarsh1986theory}, as following:
\begin{align}
    \int_{\alpha}^{\beta} (N_{h}(\tau,T)- N_{h}(\tau,H)) \text{d}\tau = - \frac{1}{2 \pi i}  \int_{\mathcal{R}} \log h(s) \text{d}s.
\end{align}
Combining the above equation with \eqref{Litlemma}, we write
\begin{align}
   & N(\sigma,T) \leq \frac{1}{2 \pi (\sigma-\alpha)} \Bigg( \int_{H}^{T} \log | h( \alpha + it)| \text{d} t \nonumber \\
    & \hspace{0.5cm} + \int_{\alpha}^{\beta} \arg h( \tau +i T) \text{d} \tau - \int_{\alpha}^{\beta} \arg h( \tau +i H) \text{d} \tau -\int_{H}^{T} \log | h( \beta + it)| \text{d} t \Bigg).\label{Littlewoodmain}
\end{align}
We apply this lemma with a suitable choice of the function $F(s)$. Here the first integral contributes significantly when $T$ is sufficiently large. We will discuss more about the contribution of these integrals in $\S$ \ref{zeroden}.
\section{Preliminaries lemmas \label{presect}} 
To begin with, we require some explicit estimates that involve sums over divisor functions in different settings. Let $X\geq X_{0}$ be a parameter which will be chosen in terms of $T$ later on. Let $\gamma$ be the Euler's constant.

\begin{lemma}
 Let $\tau>1$. Then for $X \geq 433$,
\begin{align}
\sum_{n \geq X} \frac{d(n)^2}{n^\tau} \leq \frac{ \tau}{4X^{\tau-1}}\left(\frac{(\log X)^3}{\tau-1}+\frac{3 \log ^2 X}{(\tau-1)^2}+\frac{6 \log X}{(\tau-1)^3}+\frac{6}{(\tau-1)^4}\right).\label{CHTdiv}
\end{align}
\end{lemma}
\begin{proof}
By partial summation, we have
\begin{align}
    \sum_{n \geq X}  \frac{d(n)^2}{n^{\tau}} \leq \tau \int_{X}^{\infty} \frac{\sum_{n \leq t} d(n)^{2}}{t^{\tau+1}} \text{d}t.
\end{align}
For $t \geq 433$, using the bound $\sum_{n \leq t} d(n)^{2} \leq \frac{1}{4}t \log^{3} t$ obtained by Cully-Hugill and Trudgian in \cite{cully2021two}, and then upon integrating, we get the result.

\end{proof}

In particular, for $X \geq X_{0}$ and $0.6 \leq \sigma \leq 2/3$, we have
\begin{align}
    \sum_{n \geq X} \frac{d(n)^2}{n^{2 \sigma}} \leq C(  \sigma,X_{0}) X^{1-2 \sigma} \log^{3}X, \label{CHTconcise}
\end{align}
where 
\begin{align}
    C( \sigma,X_{0})& = \frac{\sigma}{2(2 \sigma-1)^{4}} \left( (2\sigma-1)^3+\frac{3(2\sigma-1)^2 }{\log X_{0}}+\frac{6(2 \sigma-1) }{\log^{2} X_{0}}+\frac{6}{\log^{3} X_{0}}\right). \label{Csigma}
\end{align}

One of the crucial elements required here is a weaker version of the approximate functional equation for $\zeta(s)$. On choosing some specific parameters, the classical identity is as follows:
\begin{lemma}[\cite{Simoni2019ExplicitZD}]
Let $s=\sigma+it$ with $\sigma \geq 1/2$ and $t \geq 14.1347$. Then
\begin{align}
    \zeta(s)= \sum_{1 \leq n\leq t} \frac{1}{n^{s}} + R(s) \label{AFE}
\end{align}
    with $|R(s)|\leq 1.755t^{-\sigma}$.
    \end{lemma} 

Next, we prove a mean value estimate for Dirichlet polynomials that leads to an upper bound of the first integral in \eqref{Littlewoodmain}.

    \begin{lemma}\label{lemmawithu} Let $X $ be a real parameter and $m_{0}=\sqrt{1+ \frac{2}{3} \sqrt{\frac{6}{5}}}$. The following estimate holds:
    \begin{align}
        \int_{ T/2}^T & \left|\sum_{1< m\le T}m^{-s}  \sum_{1 < n \leq X}\frac{\mu(n)}{n^s}\right|^2 \text{d} t  \leq 2 \pi m_{0}  K(\sigma,T) (XT)^{2-2\sigma} \log^{2} (XT) \nonumber \\
       & + 1.36  \left(0.5  + \frac{2 \pi m_{0}}{T}\right)C(\sigma,X_{0})T X^{1 - 2\sigma}\log^3X,
       \label{Meanv}
    \end{align} 
    where $C(\sigma,X_{0})$ and $K(\sigma,T)$ are defined in \eqref{Csigma} and \eqref{Kcoeff}, respectively. 
    \end{lemma}

    \begin{proof}
    First, consider the following product of two finite sums
    \begin{align*}
        \sum_{1< m\le T}m^{-s} \sum_{1 < n \leq X}\frac{\mu(n)}{n^s}= \sum_{1 \leq n \leq XT} \frac{a_{n}(X)}{n^{s}},
    \end{align*}
    where 
    \[
a_n(X) = \sum_{\substack{e | n\\ e \le X}}\mu(e) \leq \begin{cases}
0&\text{if } n \le X\\
d(n)&\text{if }n > X.
\end{cases}
\]
Now we implement Montgomery and Vaughan's mean value theorem for Dirichlet polynomials in the form derived by Ramar\'{e} in \cite{Ram2015}, with $u_{n}=a_{n}(X)n^{-\sigma}$, as following:
    \begin{align}
    \int_{T/2}^{T} | \sum_{n} u_{n} n^{it}|^{2} \text{d}t \leq \sum_{ n } |u_{n}|^{2} \left( \frac{T}{2} + 2 \pi m_{0} (n+1)\right). \label{MVmean}
    \end{align}
     On invoking the bound \eqref{CHTconcise}, we get
    \begin{align*}
        \sum_{n} |u_{n}|^{2} (0.5 T + & 2 \pi m_{0})
         \leq (0.5 T + 2 \pi m_{0})\sum_{n \geq X} \frac{|d(n)|^{2}}{n^{2\sigma}}\\
        & \leq 1.36  \left(0.5 + \frac{2 \pi m_{0}}{T}\right)C(\sigma, X_{0})T X^{1 - 2\sigma}\log^3X.
    \end{align*}
For the other part of the sum, we use the partial summation formula
    \begin{align*}
         & \sum_{X \leq n \leq XT}  \frac{d(n)^{2}}{n^{2\sigma-1}} \nonumber \\& = (XT)^{1-2\sigma} \sum_{n \leq XT} d(n)^{2} -X^{1-2\sigma} \sum_{n \leq X} d(n)^{2}+ (2\sigma-1)\int_{X}^{XT} \frac{\sum_{n \leq t} d(n)^{2} }{t^{2 \sigma}} \text{d}t \nonumber\\
        & \leq \frac{1}{4}(XT)^{2-2\sigma} \log^{2}(XT)- 4X^{2-2\sigma} + \frac{(2\sigma-1)}{4}\int_{X}^{XT} \frac{\log^{2}t }{t^{2 \sigma-1}} \text{d}t \nonumber \\
        & \leq \left(\frac{1}{4} - \frac{4}{ T^{2-2\sigma}\log^{2} (XT)} - \frac{(2\sigma-1)}{4(2\sigma-2)}  \left( 1-\frac{1}{T^{2-2\sigma}}\right)\right)(XT)^{2-2\sigma} \log^{2} (XT) \nonumber \\
        &\leq K(\sigma,T) (XT)^{2-2\sigma} \log^{2} (XT), 
    \end{align*}
    where 
    \begin{align}
    K(\sigma,T)=\frac{1}{4} - \frac{(2\sigma-1)}{4(2\sigma-2)}\left( 1-\frac{1}{T^{2-2\sigma}}\right).\label{Kcoeff}
    \end{align}
Combining these estimates completes the proof.
\end{proof}

\section{Upper bound for the second moment of  mollified zeta function \label{Mollsect}}

Let $\sigma \geq 0.6$ and $T \geq T_{0} \geq 3 \cdot 10^{12}.$ Now, consider the mollifier
\begin{align}
    M_{X}(s)= \sum_{n \leq X} \frac{\mu(n)}{n^{s}},
\end{align}
where $\mu(n)$ is M\"{o}bius function. Now, let 
$ f_{X}(s) = \zeta(s) M_{X}(s) -1$. Choose $h=h_{X}$ as following:
\begin{align}
    h_{X}(s)= 1- f^{2}_{X}(s) =\zeta(s) M_{X}(s) ( 2- \zeta(s) M_{X}(s)). 
    \end{align}
Observe that zeros of $\zeta(s)$ are zeros of $h_{X}(s)$. Here the function $F(s)$, mentioned in $\S$ \ref{Litsect}, is $ M_{X}(s)( 2- \zeta(s) M_{X}(s))$. Now
\begin{align*}
f_X(s)&= \left(\sum_{1\le m\le T}m^{-s}+R(s) \right)\sum_{n \le X}\frac{\mu(n)}{n^s} - 1 \nonumber \\
&= \sum_{1< m\le T}m^{-s} \sum_{n < X}\frac{\mu(n)}{n^s} + R(s) \sum_{n \le X}\frac{\mu(n)}{n^s}  \nonumber \\
&=:A+B.
\end{align*}

In order to obtain bounds for the second moment of $f_{X}(s)$, write $|A + B |^2\leq 2|A|^2 + 2|B|^2$, and hence
\begin{align}
\int_{T/2}^T|f_X(s)|^2\text{d}t \leq 2 \int_{T/2}^T|A|^2 \text{d} t+ 2 \int_{ T/2}^T |B|^2 dt . \label{MVT}
\end{align} 
Now we compute the contributions of the above two integrals seperately. In Lemma \ref{lemmawithu}, the mean value estimate of $A$ is established. The two terms in \eqref{Meanv} are of the same order if $X=T^{2 \sigma-1}\log T.$ On rewriting it
\begin{align*}
    \int_{ T/2}^{T} \left|  \sum \frac{a_{m}(X)}{m^{s}}\right|^{2} dt &  \leq 1.36  \left(0.5 + \frac{2 \pi m_{0} }{T_{0}}\right) C(\sigma,X_{0}) T^{4 \sigma( 1-\sigma)}\log^3 (T^{2\sigma-1} \log T )\\
        & +  2 \pi m_{0}  K(\sigma,T) T^{4 \sigma(1-\sigma)} \log^{2}( T^{2\sigma}\log T)\\
        & \leq C_{1}( \sigma, T_{0}) T^{4\sigma(1-\sigma)}\log^{4-2\sigma}T,
\end{align*}
where
\begin{align}
    C_{1}( \sigma, T_{0}) & = 1.36  \left(0.5 + \frac{2 \pi m_{0}}{T_{0}}\right) \left(2 \sigma-1 + \frac{\log \log T_{0}}{\log T_{0}}\right)^{3}C(\sigma,T^{2 \sigma-1}_{0}) \nonumber\\
    &+ 2 \pi m_{0}  \frac{K(\sigma,T) }{\log T_{0}}\left(2 \sigma + \frac{\log \log T_{0}}{\log T_{0}}\right)^{2}.
\end{align}
Note that here we get the exponent of logarithm factor as $4-2\sigma$ whereas Carlson obtained as $3$ by comparing the terms in \eqref{Meanv} on ignoring the log factors.

Further, in \eqref{AFE}, we have $ |R(s)| \leq 1.755 t^{-\sigma}$ and hence
\begin{align*}
     \int_{T/2}^{T} |B|^{2} dt
     & \leq 3.09\int_{T/2}^{T}  t^{-2\sigma} \left( \sum_{n \le X}\frac{1}{n^\sigma} \right)^{2} dt \\
     & =  3.09\left( \frac{1}{1- \sigma} X^{1-\sigma} \right)^{2} \frac{1}{ 2\sigma-1}  \left(\frac{1}{2^{1-2 \sigma}}- 1\right)T^{1- 2 \sigma}\\
     &\leq \frac{83.17}{2 \sigma-1} \left(2^{2 \sigma-1}-1\right) T^{-(2 \sigma-1)^{2}} \log^{1-\sigma} T.
\end{align*}
Observe that the above integral contributes negligibly towards the estimation of the second moment of $f_{X}(s)$. Putting all these estimates together in \eqref{MVT}, we deduce 
\begin{align}
\int_{T/2}^T|f_X(s)|^2\text{d}t
&\leq C_{2}(\sigma, T_{0}) T^{4 \sigma(1-\sigma)} \log^{3}T, \label{lsum}
\end{align}
where $$C_{2}(\sigma, T_{0})= 2C_{1}(\sigma, T_{0}) + \frac{166.34}{(2 \sigma-1)} \frac{\left(1-2^{1-2 \sigma}\right) }{T_{0} \log^{3}T_{0}}.$$

Now, on replacing $T$ by $\frac{1}{2}T, \frac{1}{2^{2}} T, \dots$ in \eqref{lsum} and adding, we get
\begin{align}
    \int_{H}^T|f_X(s)|^2\text{d}t
&\leq \frac{C_{2}(\sigma, T_{0})}{1-0.5^{4 \sigma(1-\sigma)}} T^{4 \sigma(1-\sigma)} \log^{3}T \nonumber \\
& =: C_{3}(\sigma,T_{0})T^{4 \sigma(1-\sigma)} \log^{3}T.
\label{MVTM}
\end{align}
Here the constant $C_{3}(\sigma,T_{0})$ tends to $\frac{0.68 \sigma (2 \sigma-1)^{2}}{1-0.5^{4 \sigma(1-\sigma)}}$ when $T_{0}$ is large.


\section{Proof of Zero Density Result \label{zeroden}}

The remainder of this paper will be devoted to the proof of Theorem \ref{mainth}. Now we are in a position to bound the integrals present in \eqref{Littlewoodmain} in order to get the estimate for $N(\sigma,T)$.
 
From the definitions of $h_{X}$ and $f_{X}$, we have $|h_{X}(s)| \leq \log \{1 +|f_{X}(s)|^{2}\} \leq |f_{X}(s)|^{2}$, so that combining it with \eqref{MVTM}, we get
\begin{align}
    \int_{H}^{T} \log |h_{X}( \alpha+it)| \text{d}t \leq  C_{3}(\sigma,T_{0})T^{4 \sigma(1-\sigma)} \log^{3}T,
\end{align}
where $C_{3}(\sigma,T_{0})$ is defined in the previous section.

For the rest of the three integrals, we make use of some explicit results obtained by Kadiri, Lumley, and Ng \cite{KADIRI201822}. To deal with the integals containing the argument function, on setting $\beta=1.25$ and choosing $\alpha=\sigma-\frac{1}{\log T}$, we apply the following lemma:
 
\begin{lemma}[\cite{KADIRI201822} Lemma 4.12]
    Let $0 < H \leq H_{0} \leq T $ and $ X \leq T$. Let $ \eta=0.2561$ with $ \eta_{0}= 0.23622\cdots, \alpha$ and $\beta$ satisfying $ \frac{1}{2} \leq \alpha < 1 < \beta \leq 1 + \eta.$ Then, on fixing some parameters, we have 
    \begin{align*}
        \left| \int_{\alpha}^{\beta} \arg h_{X} ( \sigma + iT) \text{d}\sigma - \int_{\alpha}^{\beta} \arg h_{X} ( \sigma + iH ) \text{d}\sigma\right| \leq 17.29 \left( 1.25- \sigma + \frac{1}{\log T}\right) \log T.
    \end{align*}
\end{lemma}
\begin{remark}

For the case $f(s)$ to be Dirichlet $L$-functions, the bound for the $\arg$ function is provided by McCurley \cite{McCur84}. Moreover, Kadiri and Ng \cite{Kadiri2018ExplicitZD} worked it out for the Dedekind zeta function case.
\end{remark}

Now there remains to consider the fourth integral, for which we have:

\begin{lemma}[\cite{KADIRI201822} Lemma 4.13]
Let $\beta \geq 1 + \eta_{0}$ where $\eta_{0}=0.23622\cdots$. Let $X=hT$ where $T \geq H_{0}$, and $ X_{0}=hH_{0}$. Then
\begin{align}
    - \int_{H}^{T} \log | h_{X}(\beta +it)| \text{d}t \leq 0.05 \log T.
\end{align}
\end{lemma}
Now, we are all set to compile our previous bounds to obtain upper bounds for $N(\sigma,T).$
Invoking the above estimates in the equation \eqref{Littlewoodmain} yields
\begin{align*}
& N(\sigma,T)\\
& \leq \frac{\log T}{2 \pi } \left(C_{3}(\sigma,T_{0})T^{4 \sigma(1-\sigma)} \log^{3}T +17.29 \left( 1.25-\sigma+ \frac{1}{\log T}\right) \log T + 0.05 \log T\right)\\
    &\leq \left(  \frac{C_{3}(\sigma,T_{0})}{2 \pi } + \frac{17.29}{2 \pi \log^{2}T_{0}} + \frac{0.05}{4 \pi \log^{2}T_{0}}\right) T^{4 \sigma (1-\sigma)} \log^{4}T \\
    & \leq K(\sigma,T_{0})  T^{4 \sigma (1-\sigma)} \log^{4}T,
\end{align*}
where
\begin{align}
    K(\sigma,T_{0})=\left(\frac{C_{3}(\sigma,T_{0})}{2 \pi} + 0.004  \left( 1.25- \sigma + \frac{1}{\log T}\right) + \frac{0.05}{4 \pi \log^{2}T_{0}}\right). \label{const}
\end{align}

\section{Future Work}
Here, the involvement of the classical approximate functional equation \eqref{AFE} as a crucial tool makes the Carlson method less complicated than the most of the other methods. This functional equation can be easily generalized for other $L$-functions such as Dirichlet $L$-functions, Dedekind zeta functions, etc. One more ingredient of this method is the Littlewood's classical method to count the zeros for which sufficient tools have been provided to generalize for Dirichlet $L$-functions by McCurley in \cite{McCur84} and for Dedekind zeta functions, by Kadiri and Ng in \cite{Kadiri2018ExplicitZD}. Therefore, this method is easier to extend to other $L$-functions.

Titchmarsh \cite{titchmarsh1986theory} made a remark about the functional equation \eqref{AFE} not being sufficient for the estimation of the second moment of $\zeta(s)$. Whereas, combining \eqref{AFE} with the mean value estimate \eqref{MVmean} given by Montgomery and Vaughan yields

\begin{align*}
    \int_{T}^{2T} | \zeta( 1/2 +it)|^{2} \text{d}t & \leq  T \log T + 26.48T + 8.27\log T+ 17.20 + \frac{8.27}{T}.
\end{align*}
Better estimates are available these days, due to more advanced tools. 

On the other hand, when the same functional equation is applied to obtain the fourth moment of $\zeta(s)$ on the critical line, it gives 
\begin{align*}
    \int_{T}^{2T} | \zeta( 1/2 +it)|^{4} \text{d}t  \leq & 24 T^{2}\log^{3} (8T)  + 1022 T^{2} \log^{2}(2T)+ 2T \log^{3} (4T) \\
    & + 1181.16 T \log^{2}T + 19.86 T \log T + 364.25 T\\
    & +177.07 \log T + 355.83 + \frac{181.83}{T}.
\end{align*}
One can investigate a refined version which is one of the main ingredients for the zero  density estimate proved by Ingham in \cite{In40}.

\section{Acknowledgements}
The author would like to thank her supervisor Timothy S. Trudgian for encouraging her to prove the main theorem and for his continuous support. The author would also like to thank Chiara Bellotti and Andrew Yang for sharing useful unpublished results. She wants to extend her thankful gesture to Bryce Kerr and Nicol Leong for having some fruitful discussions. 

\printbibliography
\end{document}